\documentclass[12pt]{amsart}

\usepackage{amssymb,latexsym,amsmath,extarrows, mathrsfs, amsthm}
\usepackage{graphicx}

\usepackage[margin=1in, centering]{geometry}

\usepackage{color}
\usepackage{wasysym}

\usepackage{float}

\newtheorem{theorem}{Theorem}[section]
\newtheorem{lemma}[theorem]{Lemma}
\newtheorem{proposition}[theorem]{Proposition}

\newcommand{\ZZ}{\mathbb{Z}}
\newcommand{\ZN}{\mathbb{N}}
\newcommand{\ZQ}{\mathbb{Q}}

\newcommand{\Gk}{{{\rm gcd}_k }}

\allowdisplaybreaks

\title{Visible lattice points along curves}
\author{Kui Liu}
\address{School of Mathematics and Statistics, Qingdao University, 308 Ningxia Road, Shinan District, Qingdao, Shandong, China}
\email{liukui@qdu.edu.cn}

\author{Xianchang Meng}
\address{Mathematisches Institut,
	Georg-August Universit\"{a}t G\"{o}ttingen,
	Bunsenstra{\ss}e 3-5,
	D-37073 G\"{o}ttingen,
	Germany}
\email{xianchang.meng@uni-goettingen.de}


\begin{document}
	\begin{abstract}
		This paper concerns the number of lattice points in the plane which are visible  along certain curves to all elements in some set $S$ of lattice points simultaneously. By proposing the concept of level of visibility,  we are able to analyze more carefully about both the  "visible" points and the "invisible" points  in the definition of previous research. 
			We prove asymptotic formulas for the number of lattice points in different levels of visibility. 
		
	\end{abstract}
	\maketitle
	
	\section{Introduction}
	\subsection{Background}
	A lattice point $(m,n)\in\mathbb{N}\times \mathbb{N}$ is said to be visible to the lattice point $(u,v)\in\mathbb{N}\times\mathbb{N}$ along lines if there are no other integer lattice points on the straight line segment joining $(m, n)$ and $(u, v)$. In 1883, it was showed by Sylvester \cite{S} that the proportion of lattice points that are visible to the origin $(0,0)$ is $1/\zeta(2)=6/\pi^2\approx 0.60793$, where $\zeta(s)$ is the Riemann zeta function. Since then, the study of the distribution of visible lattice points continues to intrigue mathematicians till now. For example, one may refer to Adhikari-Granville \cite{A-Granville}, Baker \cite{Baker}, Boca-Cobeli-Zaharescu \cite{Boca-C-Zaharescu},  Chaubey-Tamazyan-Zaharescu \cite{Chaubey-Zaharescu}, Chen \cite{Chen}, Huxley-Nowak \cite{Huxley-N}  for part of related works and some generalizations in recent years.
	
	In 2018, Goins, Harris,
	Kubik and Mbirika \cite{GHKM} considered the integer lattice points in the plane which are visible to the origin $(0,0)$ along curves $y=r x^k$ with $k\in\mathbb{N}$ fixed and some $r\in\mathbb{Q}$. They showed that the proportion of such integer lattice points is $1/\zeta(k+1)$. In the same year, Harris and Omar \cite{HO} futher considered the case of rational exponent $k$. Recently, Benedetti, Estupi\~{ n}\'{a}n and Harris \cite{BEH}  studied the proportion of visible lattice points to the origin along such curves in higher dimensional space.
	
	All the above results are concerned about the lattice points visible to only one base point. It is natural to consider the distribution of lattice points which are visible to more base points simultaneously. For the case of visibility along straight lines, the earliest work originates from Rearick in 1960s. In his Ph.D. thesis, Rearick \cite{R-thesis} first showed that the density of integer lattice points in the plane which are jointly visible along straight lines to $N$ ($N=2$ or $3$) base points is $\prod_p\left(1-N/p^2\right)$, where the base points are  mutually visible in pairs and the product is over all the primes. Then in \cite{R}, he generalized this result to lattice points in higher dimensional space and larger $N$.
	
	The joint visibility of lattice points along curves has not been considered yet. In this paper, we focus on this topic and we also propose the concept of level of visibility. \textbf{Level-1} visibility matches the definition of "$k$-visible" in \cite{GHKM}. We use higher level of visibility to analyze more carefully about the "invisible" points  along certain curves.
	We give asymptotic formulas for the number of lattice points which are visible to a set of $N$ base points along certain curves in different levels of visibility.

	\subsection{Our results}
	For any positive integer $k$ and integer lattice points $(u,v),(m,n)\in\mathbb{N}\times\mathbb{N}$, let $r\in\mathbb{Q}$ be given by $n-v=r(m-u)^k$ and $\mathcal{C}$ be the curve $y-v=r(x-u)^k$. If there is no integer lattice points lying on the segment of $\mathcal{C}$ between points $(m,n)$ and $(u,v)$, we say $(m,n)$ is {\bf (Level-1) $k$-visible} to $(u,v)$. Further, if there is at most one integer lattice points lying on the segment of $\mathcal{C}$ between points $(m,n)$ and $(u,v)$, we say point $(m,n)$ is {\bf Level-2 $k$-visible} to $(u,v)$.
	
	One can see that $k$-visibility is mutual. Precisely, if a point $(m,n)$ is \textbf{Level-1} or \textbf{Level-2} $k$-visible to the point $(u,v)$ along the curve $y-v=r(x-u)^k$, then $(u,v)$ is also \textbf{Level-1} or \textbf{Level-2} $k$-visible to $(m,n)$, respectively, along the curve $y-n=(-1)^{k+1}r(x-m)^k$.
	
	Throughout this paper, we always assume $\boldsymbol{S}$ is a given set of integer lattice points in the plane. We say an integer lattice point $(m,n)$ is \textbf{Level-1} $k$-visible  to $\boldsymbol{S}$  if it belongs to the set
	$$
	\boldsymbol{V}_k^1(\boldsymbol{S}):=\{(m,n)\in\ZN\times\ZN:  (m,n) ~\text{is}~k\text{-visible to every point in}~ \boldsymbol{S} \}.
	$$
	Similarly, we say a point $(m,n)\in \ZN\times\ZN$ is \textbf{Level-2} $k$-visible to $\boldsymbol{S}$ if it belongs to the set
	$$
	\boldsymbol{V}_k^2(\boldsymbol{S}):=\{(m,n)\in\ZN\times\ZN:  (m,n) ~\text{is Level-2}~k\text{-visible to every point in}~ \boldsymbol{S} \}.
	$$
	One may define higher \textbf{Level} $k$-visible points to $\boldsymbol{S}$ this way.
	But in this paper, we focus on  \textbf{Level-1} and \textbf{Level-2} $k$-visible points.
	
	For $x\geq 2$, we consider visible lattice points along curves in the square $[1,x]\times [1,x]$. Denote
	\begin{align*}
	N^1_k(\boldsymbol{S},x):=&\#\{(m,n)\in \boldsymbol{V}_k^1(\boldsymbol{S}): m,n\leq x\}, 
	\end{align*}
	and
	\begin{align*}
	N^2_k(\boldsymbol{S},x):=&\#\{(m,n)\in\boldsymbol{V}_k^2(\boldsymbol{S}): m,n\leq x\}.
	\end{align*}
	
	An important case is that the points of $\boldsymbol{S}$ are pairwise $k$-visible to each other. The cardinality of such $\boldsymbol{S}$ can't be too large. In fact, we have $\#\boldsymbol{S}\le 2^{k+1}$ by Proposition \ref{prop: cardinality of sets whoses elements are muturally k-visible} in the next section.  For such type of $\boldsymbol{S}$, we obtain the following asymptotic formulas for $N^1_k(\boldsymbol{S},x)$ and $N_k^2(\boldsymbol{S},x)$.
	
	\begin{theorem}\label{thm: Level 1 density}
		 Assume the elements of $\boldsymbol{S}$ are pairwise $k$-visible to each other and $N=\#\boldsymbol{S}< 2^{k+1}$. For any $k\geq 2$, we have
		\begin{align}\label{eq: Theorem for N_k^1(S,x)}
		N^1_k(\boldsymbol{S}, x)=x^2\prod_p \Bigg(1-\frac{N}{p^{k+1}} \Bigg)+E_1(x),
		\end{align}
		where $p$ runs over all primes, and
		$$ E_1(x)=\begin{cases}
		O_k(x\log^N x),& \text{if}~1\le N\le k;\\
		O_{k,\varepsilon}(x^{2-\frac{2k}{N+k}+\varepsilon}), &\text{if}~k<N<2^{k+1}. 
		\end{cases}  $$
		
	\end{theorem}
	\textbf{Remark 1.}
	If $N=\#\boldsymbol{S}=2^{k+1}$, by Proposition \ref{prop: cardinality of sets whoses elements are muturally k-visible}, there is no lattice point outside $\boldsymbol{S}$ which is (\textbf{Level-1}) $k$-visible to all elements of $\boldsymbol{S}$.

	By the above Theorem, the density of (\textbf{Level-1}) $k$-visible points to every elements of $\boldsymbol{S}$ is $\prod_p \left(1-N/p^{k+1}\right)$. For $k=1$, it is done by the work in \cite{R}, where the author studied the visible points along straight lines. The special case $N=1$ for $k\ge 2$ in Theorem \ref{thm: Level 1 density} covers the result in \cite{GHKM}, where only the main term was given.

	We also give asymptotic formulas for \textbf{Level-2} $k$-visible points.  Note that such set actually includes some "invisible" points in the definition of previous research. We are able to analyze more carefully about these "invisible" points. 
	\begin{theorem}\label{thm: Level 2 density}
		Assume the elements of $\boldsymbol{S}$ are pairwise $k$-visible to each other and $N=\#\boldsymbol{S}\le  2^{k+1}$. For any $k\geq 1$, we have 
		\begin{align}\label{eq: Theorem for N_k^2(S,x)}
		N_k^2(\boldsymbol{S}, x)=  N_k^1(\boldsymbol{S}, x)+x^2 \frac{N}{2^{k+1}}\left(1-\frac{1}{2^{k+1}}\right)\prod_{p>2~{\rm prime}} \bigg( 1-\frac{N}{p^{k+1}}\bigg)+E_2(x),
		\end{align}
		where
		$$ E_2(x)=\begin{cases}
		O_k(x\log^N x),& \text{if}~1\le N\le k;\\
		O_{k,\varepsilon}(x^{2-\frac{2k}{N+k}+\varepsilon}), &\text{if}~k<N\le 2^{k+1}.
		\end{cases}   $$
		
	\end{theorem}
\textbf{Remark 2.}  Note that when $N=2^{k+1}$, $N_k^1(\boldsymbol{S}, x)=0$, there is no \textbf{Level-1} $k$-visible points to $\boldsymbol{S}$. But there are still positive proportion of lattice points in the plane which are \textbf{Level-2} $k$-visible to  $\boldsymbol{S}$.

For the special case $N=1$ and $k=1$, our problem is the same as the so-called "primitive lattice problem" inside a square. Nowak \cite{Nowak}, Zhai \cite{Zhai} and Wu \cite{Wu} have studied the number of primitive lattice points inside a circle. Primitive lattice points in general planar domains have also been studied by Hensley \cite{Hensley}, Huxley and Nowak \cite{Huxley-N} and Baker \cite{Baker} etc. Assuming the Riemann hypothesis(RH), they continuously improved the error term of the concerned asymptotic formulas by estimating certain exponential sums. One may wonder how much we can do to improve the estimates of $E_1(x)$ and $E_2(x)$ by similar argument under RH. However, we do not focus on pursuing the best possible error term in this paper.

	Taking $\boldsymbol{S}=\{ (0,0), (1,1)\}$, we did  numerical calculations for  densities of \textbf{Level-1} and \textbf{Level-2} $k$-visible points for $x=10000$ and $k=2, 3, \ldots, 9$ (See Table \ref{table-S-2} and Figure \ref{figure-S2} below).  We see that the numerical results match the theoretical predictions very well.

	\begin{table}[H]
		\centering
		\begin{tabular}{|c|c c|| c c|} 
			\hline
			&\multicolumn{2}{c||}{Level-1 } & \multicolumn{2}{c|}{Level-2}\\
			\hline
			$k$ & Numerical & Theoretical & Numerical & Theoretical\\ 
			\hline
			2 & 0.67680152 & 0.67689274 & 0.87422663 & 0.87431979\\ 
			3 & 0.84972063 & 0.84973299 & 0.96357826 & 0.96353652\\
			4 & 0.92895008 & 0.92905919 & 0.98893214 & 0.98906093\\
			5 & 0.96584343 & 0.96595054 & 0.99649707 & 0.99662336\\
			6 & 0.98333499 & 0.98344709 & 0.99888344 & 0.99893540\\ 
			7 & 0.99173415 & 0.99187962 & 0.99953337 & 0.99965918\\
			8 & 0.99583374 & 0.99599147 & 0.99973335 & 0.99988969\\
			9 & 0.99790020 & 0.99801286 & 0.99980001 & 0.99996401\\
			\hline
		\end{tabular}
		\caption{Densities of $k$-visible points to set of two elements}
		\label{table-S-2}
	\end{table}
	
	\begin{figure}[h]
		\centering
		\includegraphics[width=0.65\textwidth]{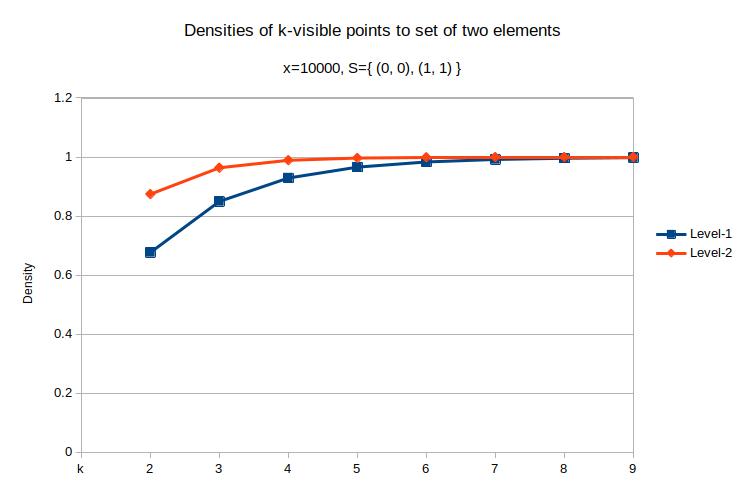}
		\caption{Densities of $k$-visible points to set of two elements}
		\label{figure-S2}
	\end{figure}
	
	We also calculate the case when $\boldsymbol{S}=\{ (0,0), (1,2), (2,1) \}$,  and we get the following data for  densities of\textbf{ Level-1} and \textbf{Level-2} $k$-visible points to $\boldsymbol{S}$ (See Table \ref{table-S-3} and Figure \ref{figure-S3}). 
	\begin{table}[H]
		\centering
		\begin{tabular}{|c|c c|| c c|} 
			\hline
			&\multicolumn{2}{c||}{Level-1 } & \multicolumn{2}{c|}{Level-2}\\
			\hline
			$k$ & Numerical & Theoretical & Numerical & Theoretical\\ 
			\hline
			2 & 0.53443474 & 0.53456687 & 0.81503364 & 0.81521448\\ 
			3 & 0.77729627 & 0.77737343 & 0.94553393 & 0.94555518\\
			4 & 0.89379137 & 0.89401525 & 0.98333222 & 0.98360945\\
			5 & 0.94873357 & 0.94899382 & 0.99464610 & 0.99493640\\
			6 & 0.97490498 & 0.97518170 & 0.99822532 & 0.99840321\\ 
			7 & 0.98750246 & 0.98782124 & 0.99920012 & 0.99948878\\
			8 & 0.99365123 & 0.99398750 & 0.99950006 & 0.99983453\\
			9 & 0.99675061 & 0.99701934 & 0.99960004 & 0.99994602\\
			\hline
		\end{tabular}
		\caption{Densities of $k$-visible points to set of three elements}
		\label{table-S-3}
	\end{table}

	\begin{figure}[H]
		\centering
		\includegraphics[width=0.65\textwidth]{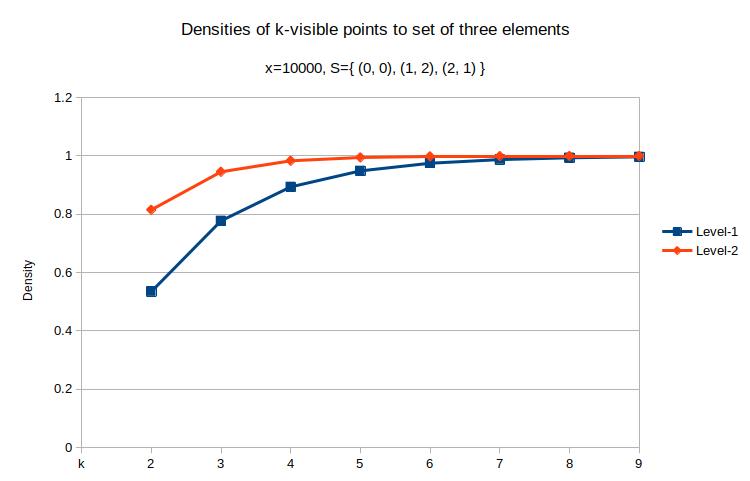}
		\caption{Densities of $k$-visible points to set of three elements}
\label{figure-S3}
	\end{figure}

	{\bf Notations.} We use $\mathbb{Z}$ to denote the set of integers; $\mathbb{N}$ to denote the set of positive integers; $\mathbb{Q}$ to denote the set of rational numbers; $\# S$ to denote the cardinality of a set $S$. As usual, we use the expressions $f=O(g)$ or $f\ll g$ to mean $|f|\leq Cg$ for some constant $C>0$. In the case when this constant $C>0$ may depend on some parameters $\rho$, we write $f=O_\rho (g)$ or $f\ll_\rho g$.

	\section{Preliminaries}
	We define the {\bf degree-$k$ greatest common divisor} of $m$, $n\in\ZZ$ as
	$$
	\Gk(m,n):=\max\{d\in\ZN: d\mid m,\ d^k\mid n \}.
	$$
	
	\begin{proposition}\label{prop: cardinality of sets whoses elements are muturally k-visible}
		For any integer $k\geq 1$, assume any two distinct elements $(u_i, v_i)$, $(u_j, v_j)\in \boldsymbol{S}$ are $k$-visible to each other, then we have $\#\boldsymbol{S}\leq 2^{k+1}$.
	\end{proposition}
	\begin{proof}
		To see this, we  consider the map
		$$\lambda: \boldsymbol{S}\rightarrow\widetilde{\boldsymbol{S}}:=\{  (u\bmod 2, v\bmod 2^{k}) : (u, v)\in \boldsymbol{S} \}.$$
		The size of the image $\widetilde{\boldsymbol{S}}$ is at most $2^{k+1}$. If $\boldsymbol{S}$ has more than $2^{k+1}$ points, there must be two distinct elements which map to the same element in $\widetilde{\boldsymbol{S}}$, say 
		$$
		\lambda((u_1, v_1) )=\lambda((u_2, v_2)).
		$$
		Thus we have
		$$2\mid (u_2-u_1) ~\text{and}~ 2^k\mid (v_2-v_1),$$
		and hence $\Gk(u_2-u_1, v_2-v_1)\geq 2$, which contradicts our assumption on $\boldsymbol{S}$.
	\end{proof}
	
	By the definition of $k$-visible points and elementary argument, we get the following lemma. One may refer to \cite{GHKM} (Proposition 3) for similar argument.  Here we omit the proof.
	\begin{lemma}\label{lem:visible criterion}
	For any $k\ge 1$, if $m-u\ne 0$ and $n-v\ne 0$, 	we have 
		\begin{enumerate}
			\item[(i)] Point $(m,n)$ is $k$-visible to point $(u,v)$ if and only if $\Gk (m-u,n-v)=1$.
			\item[(ii)] There exists exactly one integer point lying on the segment of the curve $y-v=r(x-u)^k$  joining $(u,v)$ and $(m,n)$ for some $r\in\ZQ$ if and only if $\Gk (m-u,n-v)=2$.
		\end{enumerate}
		\end{lemma}

	We  also need  the following well-known result for $l$-fold divisor function $\tau_l(n)=\sum_{d_1\cdots d_l=n} 1$. 
	\begin{lemma}[\cite{Iwaniec-Kowaski}, formula (1.80)]\label{lem: mean value of tau_k}
		Let $l\geq 2$ be an integer. For any $x\geq 2$, we have
		$$
		\sum\limits_{n\leq x}\tau_l(n)\ll_{l} x\log^{l-1} x.
		$$
	\end{lemma}

	\section{Proof of Theorem \ref{thm: Level 1 density}}\label{Proof-Thm1}
	
	Given a set $\boldsymbol{S}$, if we shift $\boldsymbol{S}$ such that it contains the origin, the error occurs to  our counting function is $O_{\boldsymbol{S}}(x)$. Thus, we may assume $(0,0)\in \boldsymbol{S}$. 
	Denote the elements of $\boldsymbol{S}$ as $(u_j, v_j),\ 0\leq j\leq N-1$ with $(u_0,v_0)=(0,0)$. 
	By Proposition \ref{prop: cardinality of sets whoses elements are muturally k-visible}, the contribution of points $(m, n)$  with $m=u_j$ or $n=v_{j'}$ for some $j, j'$  is  $O(|\boldsymbol{S}|x)=O_k(x)$. Hence, we only need to  estimate the contribution of points $(m, n)$ with $m\ne u_j$ and $n\ne v_{j'}$ for all $0\leq j, j'\leq N-1$.   
	Throughout all our proofs, we implicitly assume the input of $\Gk(*, *)$ has no zero coordinates unless otherwise specified.

	By Lemma \ref{lem:visible criterion} we have
	\begin{align*}
	N^1_k(\boldsymbol{S}, x)&=\sum_{\substack{m, n\leq x\\\Gk(m-u_j, n-v_j)=1\\ m\ne u_j, n\ne v_j\\0\leq j\leq N-1
	}} 1+O_k(x)=: \widetilde{N}^1_k(\boldsymbol{S},x)+O_k(x).
	\end{align*}
	Applying the formula
	\begin{align}\label{eq: char. function of 1}
	\sum_{d|n}\mu(d)=
	\begin{cases}
	1,\ {\rm if}\ n=1;\\
	0,\ {\rm otherwise},
	\end{cases}
	\end{align}
	where $\mu$ is the M{\"o}bius function, we write
	\begin{align}\label{eq: Expression of N_k^1(S,x)}
	\widetilde{N}^1_k(\boldsymbol{S},x)=\sum_{m, n\leq x}\sum\limits_{\substack{d_j|\gcd_k(m-u_j,n-v_j)\\0\leq j\leq N-1}}\mu(d_0)\cdots\mu(d_{N-1}).
	\end{align}

	Let $D>0$ be a parameter to be chosen later. Divide the sum over $d_0,\cdots,d_{N-1}$ into two parts: $d_0\cdots d_{N-1}\leq D$ and $d_0\cdots d_{N-1}> D$, and denote their contributions  to $\widetilde{N}^1_k(\boldsymbol{S},x)$ by $\sum_{\leq}$ and $\sum_{>}$ respectively. Then we have
	\begin{align}\label{eq: Divide N^1_k(S,x)}
	\widetilde{N}^1_k(\boldsymbol{S},x)={\sum}_{\leq}+{\sum}_{>}.
	\end{align}
	
	For $\sum_{\leq}$, we change the order of the summation and obtain
	\begin{align}
	{\sum}_{\leq}=\sum_{\substack{d_0\cdots d_{N-1}\leq D}}\mu(d_0)\cdots\mu(d_{N-1})\Bigg(\sum_{\substack{m,n\leq x\\ d_j\mid m-u_j,d_j^k\mid n-v_j\\ 0\leq j\leq N-1}}1\Bigg).
	\end{align}
	Note that $(u_0,v_0)=(0,0)$, then for any given $d_0$, $\cdots$, $d_{N-1}$, the inner sum over $m,n$ in the above formula actually equals
	$$
	\Bigg(\sum_{\substack{s\leq x/d_0\\sd_0\equiv u_j(\bmod d_j)\\1\leq j\leq N-1}}1\Bigg)\Bigg(\sum_{\substack{t\leq x/d_0^k\\td_0^k\equiv v_j(\bmod d_{j}^k)\\1\leq j\leq N-1}}1\Bigg).
	$$ 
	Since the points in $\boldsymbol{S}$ are mutually $k$-visible, then by Lemma \ref{lem:visible criterion} we have 
	$$
	\Gk(u_l-u_j, v_l-v_j)=1,\ \text{for}\  (u_l, v_l), (u_j, v_j)\in\boldsymbol{S},\ 0\leq j\neq l\leq N-1.
	$$
	This implies 
	$$
	\gcd(d_j, d_l)=1 \ \text{for}\ 0\leq j\neq l\leq N-1.
	$$ 
	It then follows that
	\begin{align*}
	{\sum}_{\leq}=&\sum_{\substack{d_0\cdots d_{N-1}\leq D\\ \gcd(d_j, d_l)=1, ~\forall\ 0\leq j\neq l\leq N-1 }}\mu(d_0)\cdots\mu(d_{N-1})\\
	& \left(\frac{x}{d_0\cdots d_{N-1}}+O(1)\right)\left(\frac{x}{d_0^k\cdots d_{N-1}^k}+O(1)\right).
	\end{align*}
	Then by Lemma \ref{lem: mean value of tau_k} we obtain
	\begin{align}
	{\sum}_{\leq}=&x^2\sum_{\substack{d_0\cdots d_{N-1}\leq D\\ \gcd(d_j, d_l)=1, ~\forall\ 0\leq j\neq l\leq N-1 }}\frac{\mu(d_0)\cdots\mu(d_{N-1})}{d_0^{k+1}\cdots d_{N-1}^{k+1}}+O\Bigg(x\sum_{d_0\cdots d_{N-1}\leq D} \frac{1}{d_0\cdots d_{N-1}} \Bigg)\nonumber\\
	&+O\Bigg(x\sum_{d_0\cdots d_{N-1}\leq D} \frac{1}{d_0^k\cdots d^k_{N-1}} \Bigg)+O\Bigg( \sum_{d_0\cdots d_{N-1}\leq D} 1 \Bigg) \nonumber \\
	=& x^2\sum_{\substack{d_0\cdots d_{N-1}\leq D\\ \gcd(d_j, d_l)=1, ~\forall\ 0\leq j\neq l\leq N-1 }}\frac{\mu(d_0)\cdots\mu(d_{N-1})}{d_0^{k+1}\cdots d_{N-1}^{k+1}}+O_k\Big(D\log^{N-1}D+x\log^{N}D\Big).
	\end{align}
	Writing $n=d_0\cdots d_{N-1}$, we then have
	\begin{align*}
	{\sum}_{\leq}=x^2\sum_{n\leq D}\frac{\mu(n)\tau_{N}(n)}{n^{k+1}}+O_k(D\log^{N-1}D+x\log^N D).
	\end{align*}
	Using  Lemma \ref{lem: mean value of tau_k}, we obtain
	\begin{align}\label{eq: Estimate for sum<=}
	{\sum}_{\leq}=x^2\sum_{n=1}^{\infty}\frac{\mu(n)\tau_{N}(n)}{n^{k+1}}+O_k \big({x^2D^{-k}\log^{N-1} D}+D\log^{N-1}D+x\log^N D\big).
	\end{align}
	
	i) If $1\leq N\leq k$, then we choose $D=x$. In this case, since each $d_j\le x^{1/k}$, $d_0\cdots d_{N-1}\le x^{N/k}\le x$. Thus the second sum $\sum_{>}$ in \eqref{eq: Divide N^1_k(S,x)} is empty since we already exclude  zero inputs of $\Gk(*,*)$ in the beginning of the proof. Inserting \eqref{eq: Estimate for sum<=} into \eqref{eq: Divide N^1_k(S,x)} yields \eqref{eq: Theorem for N_k^1(S,x)} in Theorem \ref{thm: Level 1 density} with
	$$
	E_1(x)\ll_k x\log^{N}x ~\text{for}~1\le N\le k.
	$$
	
	ii) If $k<N<2^{k+1}$, we need to make another choice for $D$, and deal with ${\sum}_{>}$ more carefully. Taking absolute value of $\mu(d)$, we obtain 
	$$
	{\sum}_{>}=\sum_{m, n\leq x}\sum\limits_{\substack{d_0\cdots d_{N-1}>D\\d_j|\gcd_k(m-u_j,n-v_j)\\0\leq j\leq N-1}}\mu(d_0)\cdots\mu(d_{N-1})\ll\sum_{m, n\leq x}\sum\limits_{\substack{d_0\cdots d_{N-1}>D\\d_j|\gcd_k(m-u_j,n-v_j)\\0\leq j\leq N-1}}1,
	$$
	which implies
	$$
	{\sum}_{>}\ll\sum_{\substack{m, n\leq x\\ \prod\limits_{0\leq j\leq N-1}\gcd_{k}(m-u_j,n-v_j)>D}}\tau\big(\Gk(m-u_0,n-v_0)\big)\cdots \tau\big(\Gk(m-u_{N-1},n-v_{N-1})\big)
	$$
	Using the bounds $\tau(n)\ll_{\varepsilon}n^{\varepsilon}$ for any $\varepsilon>0$, $N<2^{k+1}$ and $\Gk(m-u_j, n-v_j)\le x^{1/k}$,
	we have
	$$
	{\sum}_{>}\ll_{k, \varepsilon} x^{\varepsilon}\sum_{\substack{m, n\leq x\\ \prod\limits_{0\leq j\leq N-1}\Gk(m-u_j,n-v_j)>D}}1.
	$$
	Since $\prod\limits_{0\leq j\leq N-1}\gcd_{k}(m-u_j,n-v_j)>D$ implies $\Gk(m-u_{j^*},n-v_{j^*})>D^{1/N}$ for some $j^*\in\{0,\cdots,N-1\}$, we obtain
	\begin{align*}
	{\sum}_{>}\ll_{k, \varepsilon} x^{\varepsilon}\sum\limits_{0\leq j\leq N-1}\sum_{\substack{m, n\leq x\\ \Gk(m-u_j,n-v_j)>D^{1/N}}}1.
	\end{align*}
	By the definition of $\Gk$, we have
	$$
	{\sum}_{>}\ll_{k, \varepsilon} x^{\varepsilon}\sum\limits_{0\leq j\leq N-1}\sum\limits_{D^{1/N}<d\leq x^{1/k}}\sum_{\substack{m, n\leq x\\d\mid m-u_j\\d^k\mid n-v_j}}1.
	$$
	It follows that
	\begin{align}\label{eq: Estimate for sum>}
	{\sum}_{>}\ll_{k, \varepsilon} x^{2+\varepsilon}\sum\limits_{0\leq j\leq N-1}\sum\limits_{D^{1/N}<d\leq x^{1/k}}\frac{1}{d^{1+k}}\ll_{k, \epsilon} x^{2+\varepsilon}D^{-k/N}.
	\end{align}
	Collecting all the above gives
	$$
	N^1_k(\boldsymbol{S}, x)=x^2\prod_p \Bigg(1-\frac{N}{p^{k+1}} \Bigg)+O_{k,\epsilon}\big(D\log^{N-1}D+x^{2+\varepsilon}D^{-k/N}+x\log^N x\big).
	$$
	Taking $D=x^{\frac{2N}{N+k}}$ yields \eqref{eq: Theorem for N_k^1(S,x)} with 
	$$
	E_{1}(x)\ll_{k, \epsilon}  x^{2-\frac{2k}{N+k}+\varepsilon} ~\text{for}~k<N<2^{k+1}.
	$$
	
	\section{Proof of Theorem \ref{thm: Level 2 density} }
	
	In this section, we also assume $\Gk(*,*)$ does not take any input with zero coordinates. 
	
	If elements of $\boldsymbol{S}$ are pairwise $k$-visible to each other, then for any $(m,n)\in \boldsymbol{V}_k^2(\boldsymbol{S})$, there exists at most one $(u,v)\in\boldsymbol{S}$ such that $\Gk(m-u, n-v)=2$. Indeed, suppose  $\Gk(m-u_1, n-v_1)=\Gk(m-u_2, n-v_2)=2$ for some $(u_1, v_1)\neq (u_2, v_2)\in\boldsymbol{S}$, then we have
	$$2|(m-u_1), ~2|(m-u_2), ~2^k|(n-v_1), ~2^k|(n-v_2).$$
	Thus, $2|(u_2-u_1)$ and $2^k|(v_2-v_1)$, which contradicts the assumption $\Gk(u_2-u_1, v_2-v_1)=1$.

	By the above argument, we  write
	\begin{align}\label{eq: Level-2-count}
	N_k^2(\boldsymbol{S}, x)=N_k^1(\boldsymbol{S}, x)+\sum_{0\leq l\leq N-1}\sum_{\substack{ m, n\leq x\\ \Gk(m-u_l, n-v_l)=2\\ \Gk(m-u_j, n-v_j)=1\\ j\neq l 
	}}1+O_k(x).
	\end{align}
	Without loss of generality, we may assume $(u_0, v_0)=(0,0)$. We only need to estimate the inner sum of the second term in \eqref{eq: Level-2-count} for $l=0$, other cases are similar.  Denote
	\begin{align*}
	I(x):=\sum_{\substack{m, n\leq x\\ \Gk(m, n)=2\\ \Gk(m-u_j, n-v_j)=1\\ 1\leq j\leq N-1 } } 1.    
	\end{align*}
	We have
	\begin{align}\label{eq: Definition of I(x)}
	I(x)=\sum_{\substack{m, n\leq x\\2\mid m,2^k\mid n\\ \Gk(m/2, n/2^{k})=1\\ \Gk(m-u_j, n-v_j)=1\\ 1\leq j\leq N-1 } } 1=\sum_{\substack{m,n\leq x\\2\mid m,2^k\mid n}}\sum_{\substack{d_0\mid \Gk(m/2, n/2^k)\\ d_j\mid \Gk(m-u_j, n-v_j)\\ 1\leq j\leq N-1}} \mu(d_0)\cdots\mu(d_{N-1}).
	\end{align}
	By changing the order of  summation and making the substitutions $m=2d_0s$ and $n=(2d_0)^kt$, we obtain
	\begin{align}\label{eq: Main expression for I(x)}
	I(x)=\sum_{d_0, \cdots, d_{N-1}\leq x^{1/k}}\mu(d_0)\cdots\mu(d_{N-1})\sum_{\substack{s\leq x/(2d_0), t\leq x/(2d_0)^k\\2d_0 s\equiv u_j(\bmod d_j) \\ (2d_0)^k t\equiv v_j(\bmod d_j^k) \\1\leq j\leq N-1 } }1.
	\end{align}
	
	In order to get estimates of $I(x)$, we need to analyze the conditions in the inner sum. Fix $d_0, \cdots, d_{N-1}$, in order for those congruence equations having solutions, we need 
	$$
	\gcd(2d_0, d_j)\mid u_j, ~\gcd((2d_0)^k, d_j^k)\mid v_j
	$$
	for $1\leq j\leq N-1$.
	Since points $(u_j,v_j)$ are $k$-visible to point $(u_0,v_0)$, then Lemma \ref{lem:visible criterion} gives $\Gk(u_j, v_j)=1$. It follows that $
	\gcd(2d_0, d_j)=1$ for $1\leq j\leq N-1$.
	Moreover, in order for those congruence equations having solutions, we also need the following equations
	$$
	d_{j_1}l_{1}-d_{j_2}l_{2}= u_{j_2}-u_{j_1}, ~d^k_{j_1}t_{1}-d^k_{j_2}t_{2}= v_{j_2}-v_{j_1}
	$$
	have solutions for any $d_{j_1}$ and $d_{j_2}$ with $1\leq j_1\neq j_2\leq N-1$. This implies 
	$$
	\gcd(d_{j_1},d_{j_2})\mid u_{j_2}-u_{j_1},\ \gcd(d_{j_1}^k,d_{j_2}^k)\mid v_{j_2}-v_{j_1}
	$$
	for $1\leq j_1\neq j_2\leq N-1$.
	By the assumption of pairwise $k$-visibility of elements of $\boldsymbol{S}$, we have $\Gk(u_{j_2}-u_{j_1}, v_{j_2}-v_{j_1})=1$, and thus
	$\gcd(d_{j_1}, d_{j_2})=1$ for any $1\leq j_1\neq j_2\leq N-1$. 
	
	As what we did in Section \ref{Proof-Thm1}, we divide the sum over $d_0,\cdots,d_{N-1}$ into two parts according to $d_0\cdots d_{N-1}\leq D$ or not. Denote them by $I_{\leq}$ and $I_{>}$ respectively, then
	\begin{align}\label{eq: Divide I(x)}
	I(x)=I_{\leq}+I_{>}.
	\end{align}
	
	For $I_{\leq}$, we have
	$$
	I_{\leq}=\sum_{\substack{d_0\cdots d_{N-1}\leq D\\ \gcd(d_{j_1}, d_{j_2})=1, \forall j_1\neq j_2\\
			\gcd(2, d_j)=1, 1\leq j\leq N-1
	}} \mu(d_0)\cdots\mu(d_{N-1})\left(\frac{x}{2d_0\cdots d_{N-1}}+O(1)\right)\left(\frac{x}{2^kd_0^k\cdots d_{N-1}^k}+O(1)\right),
	$$
	and by Lemma \ref{lem: mean value of tau_k}, we get 
	$$
	I_{\leq}=\frac{x^2}{2^{k+1}} \sum_{\substack{d_0\cdots d_{N-1}\leq D\\ \gcd(d_{j_1}, d_{j_2})=1,\forall j_1\neq j_2\\
			\gcd(2, d_j)=1, 1\leq j\leq N-1
	}}\frac{\mu(d_0)\cdots\mu(d_{N-1})}{d_0^{k+1}\cdots d_{N-1}^{k+1}}+O_k\left(x\log^N x+D\log^{N-1} D\right).
	$$
	Making the substitution $n=d_0\cdots d_{N-1}$, we obtain
	$$
	I_{\leq}=\frac{x^2}{2^{k+1}} \sum_{n\leq D} \frac{\mu(n)}{n^{k+1}}h(n)+O_k\left(x\log^N x+D\log^{N-1} D\right).
	$$
	where 
	$$
	h(n)=\sum\limits_{\substack{n=d_0\cdots d_{N-1}\\ d_1,\cdots, d_{N-1} ~\text{odd} }} 1.
	$$
	Extending the sum over $n$ and using the bound $h(n)\leq \tau_{N}(n)$,  and by Lemma \ref{lem: mean value of tau_k}, we derive
	\begin{align}\label{eq: Estimate for I<=}
	I_{\leq }=\frac{x^2}{2^{k+1}} \sum_{n=1}^{\infty} \frac{\mu(n)}{n^{k+1}}h(n)+O_k\left(x^{2}D^{-k}\log^{N-1} D+x\log^N x+D\log^{N-1} D\right).
	\end{align}
	Note that $h(n)$ is multiplicative with $h(2)=1$ and $h(p)=N$ for $p> 2$ prime. Thus
	$$
	\sum_{n=1}^{\infty} \frac{\mu(n)}{n^{k+1}}h(n)=\bigg(1-\frac{1}{2^{k+1}} \bigg) \prod_{p>2} \bigg(1-\frac{N}{p^{k+1}} \bigg). 
	$$
	
	i) If $1\leq N\leq k$, then we choose $D=x$. In this case, since each $d_j\le x^{1/k}$, $d_0\cdots d_{N-1}\le x^{N/k}\le x$. Thus the second term $I_{>}$ in \eqref{eq: Divide I(x)} is empty. Inserting \eqref{eq: Estimate for I<=} into \eqref{eq: Divide I(x)} yields \eqref{eq: Theorem for N_k^2(S,x)} in Theorem \ref{thm: Level 2 density} with
	$$
	E_1(x)\ll_k x\log^{N}x ~\text{for}~1\le N\le k.
	$$
	
	ii) If $k<N\le 2^{k+1}$, we need to make another choice for $D$ and deal with $I_{>}$.
	By similar argument as before, we obtain
	$$
	I_{>}\ll\sum_{\substack{m,n\leq x\\2\mid m,2^k\mid n}}\sum_{\substack{d_0\cdots d_{N-1}>D\\d_0\mid \Gk(m/2, n/2^k)\\ d_j\mid \Gk(m-u_j, n-v_j)\\ 1\leq j\leq N-1}} 1,
	$$
	which gives
	$$
	I_{>}\ll\sum\limits_{\substack{m,n\leq x\\2\mid m,2^k\mid n\\ \prod\limits_{0\leq j\leq N-1}\Gk(m-u_j,n-v_j)>D}}\tau(\Gk(m/2,n/2^k))\prod\limits_{1\leq j\leq N-1}\tau(\Gk(m-u_j,n-v_j)).
	$$
	Using the bound $\tau(n)\ll_{\varepsilon} n^{\varepsilon}$ for any $\varepsilon>0$, by a similar argument as in the proof of Theorem \ref{thm: Level 1 density}, we obtain
	$$
	I_{>}\ll_{\varepsilon} x^{\varepsilon}\sum\limits_{\substack{m,n\leq x\\ \prod\limits_{0\leq j\leq N-1}\Gk(m-u_j,n-v_j)>D}}1\ll_{\varepsilon} x^{2+\varepsilon}D^{-k/N}.
	$$

	Hence, combining all the estimates and taking $D=x^{\frac{2N}{N+k}}$ yields 
	\begin{align*}
	I(x)=\frac{x^2}{2^{k+1}} \bigg(1-\frac{1}{2^{k+1}} \bigg) \prod_{p>2} \bigg(1-\frac{N}{p^{k+1}} \bigg) +O_{k,\epsilon}(x^{2-\frac{2k}{N+k}+\varepsilon}+x\log^N x).
	\end{align*}
	Plugging this into \eqref{eq: Level-2-count}, we obtain \eqref{eq: Theorem for N_k^2(S,x)} in Theorem \ref{thm: Level 2 density} with $$
	E_{2}(x)\ll_{k,\varepsilon}  x^{2-\frac{2k}{N+k}+\varepsilon} ~\text{for}~k<N\le 2^{k+1}.
	$$

	{\bf Acknowledgements.}
	The first author is partially supported by Shandong Provincial Natural Science Foundation (Grant No. ZR2019BA028).
	The second author is partially supported by the Humboldt Professorship of Professor Harald Helfgott. Both authors thank the anonymous referee for valuable suggestions.

\end{document}